\documentclass[a4paper]{article}

\usepackage{amsmath}
\usepackage{amsthm}
\usepackage{amssymb}

\usepackage[pdftex,colorlinks=true,urlcolor=blue,linkcolor=blue,plainpages=false,pdfpagelabels
]{hyperref}

\newtheorem{theorem}{Theorem}[section]

\newtheorem{definition}[theorem]{Definition}
\newtheorem{proposition}[theorem]{Proposition}
\newtheorem{corollary}[theorem]{Corollary}
\newtheorem{lemma}[theorem]{Lemma}
\newtheorem{fact}[theorem]{Remark}

\newcommand{\bdfn}{\begin{definition}}
\newcommand{\edfn}{\end{definition}}
\newcommand{\bthm}{\begin{theorem}}
\newcommand{\ethm}{\end{theorem}}
\newcommand{\bprop}{\begin{proposition}}
\newcommand{\eprop}{\end{proposition}}
\newcommand{\bcor}{\begin{corollary}}
\newcommand{\ecor}{\end{corollary}}
\newcommand{\blem}{\begin{lemma}}
\newcommand{\elem}{\end{lemma}}
\newcommand{\bfact}{\begin{fact}}
\newcommand{\efact}{\end{fact}}

\newcommand{\be}{\begin{enumerate}}
\newcommand{\ee}{\end{enumerate}}

\newcommand{\beq}{\begin{equation}}
\newcommand{\eeq}{\end{equation}}
\newcommand{\ba}{\begin{array}} 
\newcommand{\ea}{\end{array}}
\newcommand {\bea} {\begin{eqnarray}}
\newcommand {\eea} {\end {eqnarray}}
\newcommand {\bua} {\begin{eqnarray*}}
\newcommand {\eua} {\end {eqnarray*}}

\newcommand{\ds}{\displaystyle}
\def\R{{\mathbb R}}
\def\N{{\mathbb N}}
\def\Q{{\mathbb Q}}
\newcommand{\ul}{\underline}
\newcommand{\se}{\subseteq}

\newcommand{\eps}{\varepsilon}
\newcommand{\Ra}{\Rightarrow}
\newcommand{\ra}{\rightarrow}
\newcommand{\si}{\wedge}
\newcommand{\cA}{\mathcal{A}}
\newcommand{\limn}{\ds\lim_{n\to\infty}}
\newcommand{\lsupn}{\ds \limsup_{n\to\infty}}
\newcommand{\linfn}{\ds\liminf_{n\to\infty}}
\newcommand{\modliminf}{modulus of liminf }

\newcommand{\lambdaxy}{(1-\lambda)x\oplus\lambda y}
\newcommand{\aomega}{{\cal A}^{\omega}}

\begin{document}
\title{An application of proof mining to nonlinear iterations}
\author{Lauren\c{t}iu Leu\c{s}tean\\[0.2cm]
\footnotesize Simion Stoilow Institute of Mathematics of the Romanian Academy, Research unit 5,\\
\footnotesize P. O. Box 1-764, RO-014700 Bucharest, Romania\\[0.1cm]
\footnotesize E-mail: Laurentiu.Leustean@imar.ro
}

\date{}
\maketitle

\begin{abstract}
In this paper we apply methods of proof mining to obtain a highly uniform effective rate of asymptotic regularity for the Ishikawa iteration associated to nonexpansive self-mappings of convex subsets of a class of uniformly convex geodesic spaces.  Moreover,  we show that these results are guaranteed by a combination of  logical metatheorems for classical and  semi-intuitionistic systems. 
\end{abstract}

\section{Introduction}

{\em Proof mining} is a paradigm of research concerned with the extraction of hidden finitary 
and combinatorial content from proofs that make use of highly infinitary principles. 
This new information is obtained after a logical analysis of the original mathematical proof, 
using proof-theoretic techniques called {\em proof interpretations}. 
In this way one obtains highly uniform effective bounds for results that are more general 
than the initial ones. While the methods used to obtain these new results come from mathematical 
logic, their proofs can be written in ordinary mathematics. We refer to Kohlenbach's book \cite{Koh08-book} 
for a comprehensive reference for proof mining.

This line of research, developed by Kohlenbach in the 90's, has its origins in 
Kreisel's program of {\em unwinding of proofs}. Kreisel's idea was to apply proof-theoretic techniques  
to analyze concrete mathematical proofs and unwind the 
information hidden in them; see for example \cite{KreMacint82} and, more recently, \cite{Macint05}.

Proof mining has numerous applications to approximation theory, 
asymptotic behavior of nonlinear iterations, as well as (nonlinear) ergodic theory, topological 
dynamics and Ramsey theory. 
In these applications, Kohlenbach's {\em monotone} functional interpretation \cite{Koh96a} is 
crucially used, since it systematically transforms any statement in a given proof 
into a new  version for which explicit bounds are provided.

Terence Tao \cite{Tao07} arrived at a proposal of so-called {\em hard analysis} 
(as opposed to {\em soft analysis}), inspired by the finitary arguments used by him and Green \cite{GreTao08} in their 
proof that there are arithmetic progressions of arbitrary length in the prime numbers, as well as by 
him alone in a series of papers \cite{Tao06,Tao07a,Tao08,Tao08a}. 
As Kohlenbach points out in \cite{Koh07}, 
Tao's hard analysis could be viewed as carrying out, 
using monotone functional interpretation, analysis on the level of 
uniform bounds. 

For mathematical proofs based on classical logic, general logical metatheorems were obtained 
by Kohlenbach  \cite{Koh05} for important classes of metrically bounded spaces in functional 
analysis and generalized to the unbounded case by Gerhardy and Kohlenbach \cite{GerKoh08}. 
They considered metric, hyperbolic and CAT(0)-spaces, (uniformly convex) normed spaces and 
inner product spaces also with abstract convex subsets. The metatheorems were adapted to 
Gromov $\delta$-hyperbolic spaces and $\R$-trees 
\cite{Leu06}, complete metric and normed spaces  \cite{Koh08-book} and uniformly smooth Banach spaces 
\cite{KohLeu12}. The proofs of the metatheorems are based on extensions to the new formal systems of G\" odel's 
functional interpretation combined with negative translation and parametrized  versions of majorization. These logical metatheorems 
guarantee that one can extract  effective uniform bounds from classical proofs of 
$\forall\,\exists$-sentences and that these bounds are independent from parameters satisfying 
weak local boundedness conditions. Thus, the metatheorems can be used to conclude the existence of 
effective uniform bounds without having to carry out the proof analysis: we have to verify only that 
the statement has the right logical form and that the proof can be formalized in our system.

Gerhardy and Kohlenbach \cite{GerKoh06} obtained  similar logical metatheorems for proofs in 
semi-intuitionistic systems, that is proofs based on  intuitionistic logic enriched with 
noneffective principles, such as comprehension in all types for arbitrary negated or $\exists$-free formulas. 
The proofs of these metatheorems use  monotone modified realizability, a monotone version of Kreisel's modified realizability 
\cite{Kre59}. A great benefit of this setting is that there are basically no restrictions on the logical complexity of 
mathematical theorems for which bounds can be extracted. 

{\em The goal of this paper is to present an application of proof mining to the asymptotic behavior of Ishikawa 
iterations for nonexpansive mappings.}

Let $X$ be a normed space, $C\se X$ a convex subset and $T:C\to C$. We shall denote with $Fix(T)$ the set of fixed points of $T$.  The {\em Ishikawa iteration} 
starting with $x\in C$ was introduced in \cite{Ish74} as follows:
\[
x_0=x, \quad x_{n+1}=(1-\lambda_n)x_n+ \lambda_nT((1-s_n)x_n+s_nTx_n),
\]
where $(\lambda_n),(s_n)$ are sequences in $[0,1]$.  The well-known Krasnoselski-Mann iteration \cite{Kra55,Man53} is 
obtained as a special case by taking $s_n=0$ for all $n\in\N$.  

Ishikawa proved that for convex compact subsets $C$ of Hilbert spaces
and Lipschitzian pseudocontractive mappings $T$, this iteration converges strongly towards 
a fixed point of $T$, provided that the sequences $(\lambda_n)$ and $(s_n)$ satisfy some assumptions.

In the following we consider the Ishikawa iteration for nonexpansive mappings and sequences $(\lambda_n),(s_n)$ satisfying 
the following conditions:
\beq
\sum_{n=0}^\infty\lambda_n(1-\lambda_n) \text{ diverges, }  \lsupn s_n<1 \text{ and }\sum_{n=0}^\infty s_n(1-\lambda_n)
\text{ converges.} \label{hyp-lambdan-sn}
\eeq
Tan and Xu \cite{TanXu93} proved the weak convergence of the Ishikawa iteration in uniformly convex Banach spaces $X$ 
which satisfy Opial's condition or whose norm is Fr\'{e}chet differentiable, generalizing in this way a 
well-known result of Reich \cite{Rei79} for the Krasnoselski-Mann iteration. Dhompongsa and Panyanak \cite{DhoPan08} obtained 
the $\Delta$-convergence of the Ishikawa iteration in CAT(0) spaces. $\Delta$-convergence is a concept of 
weak convergence in metric spaces introduced by Lim \cite{Lim76}.

One of the most important properties of any iteration associated to a nonlinear mapping is 
asymptotic regularity, defined by Browder and Petryshyn \cite{BroPet66} for the Picard iteration.  
The {\em Ishikawa iteration}  $(x_n)$ is said to be {\em asymptotically regular }  if 
$\limn \|x_n-Tx_n\|=0$. A rate of convergence of $(\|x_n-Tx_n\|)$ towards $0$ will be called a
 {\em rate of asymptotic regularity} of $(x_n)$.
Asymptotic regularity is the first property one gets before proving the  weak or strong 
convergence  of the iteration towards a fixed point of the mapping. Thus, the following 
asymptotic regularity result is  implicit in the proof of Tan and Xu.

\bthm\label{Ishikawa-as-reg-Banach}
Let $X$ be a uniformly convex Banach space,  $C\se X$ a  convex subset and 
$T:C\to C$ be nonexpansive with $Fix(T)\ne\emptyset$.  Assume that  $(\lambda_n),(s_n)$ satisfy \eqref{hyp-lambdan-sn}. Then  $\limn \|x_n-Tx_n\|=0$ for all $x\in C$.
\ethm

In this paper we show that the proof of the generalization of Theorem \ref{Ishikawa-as-reg-Banach} to a class of uniformly convex geodesic spaces (the so-called $UCW$-hyperbolic 
spaces) can be analyzed using a combination of logical metatheorems for the classical and semi-intuitionistic setting. As we explain in Section  \ref{section-logic-discussion}, 
there are two main steps, 
the first one with a classical proof, analyzed using the combination of monotone functional interpretation and negative translation, while the second one has a constructive proof, 
analyzed more directly using monotone modified realizability.

As a consequence, the logical metatheorems guarantee that one can obtain a quantitative version for the generalization of Theorem \ref{Ishikawa-as-reg-Banach} 
obtained by taking convex  subsets of $UCW$-hyperbolic spaces and by replacing the hypothesis of $T$ having  fixed points with the weaker assumption  that $T$ has approximate fixed points 
in a $b$-neighborhood of the starting point $x$ for some $b>0$.  In the last section of the paper  (Theorem \ref{Ishikawa-main}) we give a direct mathematical proof of this quantitative version, 
providing a uniform rate of asymptotic regularity for the Ishikawa iteration.

We point out that in \cite{Leu10} we computed, for the case when we assume that $T$ has fixed points, a rate of asymptotic regularity $\Phi$ for  $(x_n)$. Applied to the analyzed proof in 
\cite{Leu10}, the logical metatheorems guarantee a priori that the same rate $\Phi$ as in \cite{Leu10} holds when we assume only the existence of approximate fixed points instead.  
However, the rate we compute in Theorem \ref{Ishikawa-main} is slightly changed because we give a more readable mathematical proof of this result.

As we use both functional and modified realizability interpretations to give a logical explanation  of our results, we think that an interesting direction of 
research could be to see if the hybrid  functional interpretation \cite{HerOli08,Oli12} can be used instead. \\[0.2cm]
{\bf Notation}: $\N=\{0,1,2\ldots, \}$ and $[m,n]=\{m,m+1,\ldots, n-1,n\}$ for any $m, n\in\N, m\leq n$.

\section{Logical metatheorems for $UCW$-hyperbolic spa-ces}\label{section-logic}

A {\em  $W$-hyperbolic space}  is a structure $(X,d,W)$, where $(X,d)$ is a metric space 
and $W:X\times X\times [0,1]\to X$  is a {\em convexity} mapping satisfying the following axioms:
\begin{eqnarray*}
(W1) & d(z,W(x,y,\lambda))\le (1-\lambda)d(z,x)+\lambda d(z,y),\\
(W2) & d(W(x,y,\lambda),W(x,y,\tilde{\lambda}))=|\lambda-\tilde{\lambda}|\cdot 
d(x,y),\\
(W3) & W(x,y,\lambda)=W(y,x,1-\lambda),\\
(W4) & \,\,\,d(W(x,z,\lambda),W(y,w,\lambda)) \le (1-\lambda)d(x,y)+\lambda
d(z,w).
\end {eqnarray*}

Takahashi \cite{Tak70} initiated in the 70's the study of {\em convex metric spaces} as 
structures $(X,d,W)$ satisfying (W1). The notion of $W$-hyperbolic space defined above was 
introduced by Kohlenbach \cite{Koh05}. We refer to \cite[p.384]{Koh08-book} for a very nice 
discussion on these spaces and related structures. First examples of $W$-hyperbolic spaces are normed spaces; just take 
$W(x,y,\lambda)=(1-\lambda )x+\lambda y$. A very important class of $W$-hyperbolic spaces are
Busemann's non-positively curved spaces \cite{Bus48,Bus55}, extensively studied  in the 
monograph \cite{Pap05}. Given $x,y\in X$ and $\lambda\in[0,1]$, we shall use the notation $(1-\lambda)x\oplus \lambda y$ for 
$W(x,y,\lambda)$.  A nonempty subset $C\subseteq X$ is said to be {\em convex} if $(1-\lambda)x\oplus \lambda y \in C$ 
for all $x,y\in C$ and all $\lambda\in [0,1]$.

Uniform convexity can be defined in the setting of $W$-hyperbolic spaces following Goebel and 
Reich's definition for the Hilbert ball \cite[p.105]{GoeRei84}. A $W$-hyperbolic space $(X,d,W)$ 
is said to  be {\em uniformly convex} \cite{Leu07} if 
there exists a mapping $\eta:(0,\infty)\times (0,2]\to (0,1]$ such that for all $r>0, 
\varepsilon\in(0,2]$ and  all $a,x,y\in X$,
\begin{eqnarray*}
\left.\begin{array}{l}
d(x,a)\le r\\
d(y,a)\le r\\
d(x,y)\ge\varepsilon r
\end{array}
\right\}
& \quad \Rightarrow & \quad d\left(\frac12x\oplus\frac12y,a\right)\le (1-\eta(r,\eps))r. \label{uc-def}
\end{eqnarray*}
The mapping $\eta$  is said to be a {\em modulus of uniform convexity}.  We 
use the notation $(X,d,W,\eta)$ for a uniformly convex $W$-hyperbolic space with modulus $\eta$.

Uniformly convex $W$-hyperbolic spaces $(X,d,W,\eta)$ with $\eta$ being nonincreasing in the first argument are called {\em $UCW$-hyperbolic spaces}, following  \cite{Leu10}. Obviously, 
uniformly convex Banach spaces are $UCW$-hyperbolic spaces with a modulus $\eta$ 
that does not depend on $r$ at all. Other examples of $UCW$-hyperbolic space are CAT(0) 
spaces, important structures in geometric group theory (see  \cite{BriHae99}). As  the 
author remarked in \cite{Leu07}, CAT(0) spaces have a modulus of uniform convexity  
$\ds\eta(\varepsilon)=\frac{\varepsilon^2}{8}$, quadratic in $\eps$.

In the following  we  give adaptations to  $UCW$-hyperbolic spaces of general logical metatheorems 
for $W$-hyperbolic spaces proved by Gerhardy and Kohlenbach for classical systems in 
\cite{GerKoh08} and for intuitionistic systems in \cite{GerKoh06}.

Let ${\cal A}^{\omega}$ be the system of {\em weakly extensional} classical analysis, which goes back to 
Spector \cite{Spe62}.  It is formulated in the language of functionals of finite types and consists of  
$\mathbf{WE-PA^\omega}$, the weakly extensional Peano arithmetic in all finite types, the axiom schema  
$\mathbf{QF-AC}$  of quantifier-free axiom of choice and the axiom schema $\mathbf{DC^\omega}$ of dependent 
choice in all finite types. Full second order arithmetic in the sense of reverse mathematics \cite{Sim09} 
is contained in ${\cal A}^{\omega}$ if we identify subsets of $\N$ with their characteristic functions. 
We refer the reader to \cite{Koh08-book} for all the undefined  notions related to the system 
${\cal A}^{\omega}$, including the representation of real numbers in this system. 
As a consequence of this representation, the relations $=_\R$, $\leq_\R$ are given by $\Pi_1^0$ predicates, 
while $<_\R$ is given by a $\Sigma_1^0$ predicate.

The theory $\cA^{\omega}[X,d]_{-b}$ for  metric spaces is defined in \cite{GerKoh08} by extending 
$\cA^{\omega}$ to the set ${\bf T}^{X}$ of all finite types over the ground types $0$ and $X$ 
and  by adding two new constants $0_{X}$ of type $X$ and  $d_X$ of type $X\rightarrow X\rightarrow 1$ 
together with axioms expressing the fact that $d_{X}$ represents a pseudo-metric. One defines the equality $=_{X}$ between 
objects of type $X$ as follows:
\bua
x=_X y\,:= \, d_X(x,y)=_\R 0_\R.
\eua
Then $d_X$ represents a metric on the set of equivalence classes generated by $=_X$.

We use the subscript $_{-b}$ here and for the theories defined in the sequel in order to be consistent 
with the notations from \cite{Koh08-book}.

The theory $\cA^\omega[X,d,W]_{-b}$  for $W$-hyperbolic spaces results from $\cA^\omega[X,d]_{-b}$ by adding  a 
new constant $W_X$ of type $X\rightarrow X\rightarrow 1 \rightarrow X$ together with the 
formalizations of the axioms (W1)-(W4).

In order to define the theory associated to $UCW$-hyperbolic spaces, we prove the following lemma, giving equivalent characterizations 
for these spaces.

\bprop\label{char-UCW-hyperbolic}
Let $(X,d,W)$ be a $W$-hyperbolic space. The following are equivalent:
\be
\item\label{uc-mod-1} $X$ is an $UCW$-hyperbolic space with modulus $\eta$.
\item\label{uc-mod-2} there exists $\eta_1:(0,\infty)\times \N\rightarrow \N$ nondecreasing in the first argument such that 
for any $r>0,k\in\N$, and $x,y,a\in X$ 
\begin{eqnarray*}
\left.\begin{array}{l}
d(x,a) \leq r\\
d(y,a) \leq r\\
d\left(\frac12x\oplus\frac12y,a\right) > \left(1-2^{-\eta_1(r,k)}\right)r
\end{array}
\right\}
& \quad \Rightarrow & \quad d(x,y) <  2^{-k}r.
\end{eqnarray*}
\item\label{uc-mod-3} there exists $\eta_2:(0,\infty)\times \N\rightarrow \N$ nondecreasing in the first argument such that 
for any $r>0,k\in\N$, and $x,y,a\in X$ 
\begin{eqnarray*}
\left.\begin{array}{l}
d(x,a)< r\\
d(y,a) < r\\
d\left(\frac12x\oplus\frac12y,a\right) > \left(1-2^{-\eta_2(r,k)}\right)r
\end{array}
\right\}
& \quad \Rightarrow & \quad d(x,y)\le 2^{-k}r.
\end{eqnarray*}
\item\label{uc-mod-4} there exists $\eta_3:\Q_*^+\times \N\rightarrow \N$ nondecreasing in the first argument such that 
for any $r\in\Q_*^+,k\in\N$, and $x,y,a\in X$ 
\begin{eqnarray*}
\left.\begin{array}{l}
d(x,a)< r\\
d(y,a) < r\\
d\left(\frac12x\oplus\frac12y,a\right) > \left(1-2^{-\eta_3(r,k)}\right)r
\end{array}
\right\}
& \quad \Rightarrow & \quad d(x,y)\le 2^{-k}r. 
\end{eqnarray*}
\ee
\eprop
\begin{proof}
$\eqref{uc-mod-1}\Ra\eqref{uc-mod-2}$ Let $\eta$ be a modulus of uniform convexity nonincreasing in the first argument
and define $\ds \eta_1:(0,\infty)\times \N\rightarrow \N, \quad\ds \eta_1(r,k)=\left\lceil -\log_2 \eta(r,2^{-k})\right\rceil.$

$\eqref{uc-mod-2}\Ra\eqref{uc-mod-1}$ Define $\ds \eta:(0,\infty)\times(0,2]\rightarrow (0,1], \quad \eta(r,\eps)=2^{-\eta_1\left(r, \left\lceil -\log_2 \eps \right\rceil\right)}.$

$\eqref{uc-mod-2}\Ra\eqref{uc-mod-3}$ Obviously,  just take  $\eta_2:=\eta_1$.

$\eqref{uc-mod-3}\Ra\eqref{uc-mod-2}$   Define $\eta_1(r,k)=\eta_2(r,k+1)$. 
Let $r>0, k\in\N$ and $a,x,y\in X$ be such that  $\ds d(x,a)\leq r, d(y,a) \leq r$ and $d\left(\frac12x\oplus\frac12y,a\right) > 
\left(1-2^{-\eta_1(r,k)}\right)r=\left(1-2^{-\eta_2(r,k+1)}\right)r$.
Define $x_n=\left(1-\frac1n\right)x\oplus \frac1n a$, $y_n=\left(1-\frac1n\right)y\oplus \frac1n a$ and $z_n=\frac12x_n\oplus\frac12y_n$ for all $n\ge 1$.
Then $d(x_n,a) = \left(1-\frac1n\right)d(x,a)< r$ and similarly $d(y_n,a)<r$.  Furthermore, 
$d(x_n,x)=\frac1n d(x,a), d(y_n,y)=\frac1n d(y,a)$, so $\ds\limn x_n=x, \limn y_n=y$. Since $0\leq d\left(z_n,\frac12x\oplus\frac12y\right)\leq \frac12d(x_n,x)+\frac12d(y_n,y)$, 
 we get that  $\ds \limn z_n=\frac12x\oplus\frac12y$. Applying the continuity  of $d$, it follows that 
\[\limn d(x_n,y_n)=d(x,y) \text{ and } \limn d(z_n,a)=d\left(\frac12x\oplus\frac12y,a\right).\]  
Hence, $d(z_n,a) > \left(1-2^{-\eta_2(r,k+1)}\right)r$ for all $n$ from some $N$ on.  We can then apply \eqref{uc-mod-3} to get that for all $n\ge N$, 
 $d(x_n,y_n) \leq 2^{-k-1}r$. By letting $n\to \infty$, it follows that $d(x,y)\leq 2^{-k-1}r<2^{-k}r$.

$\eqref{uc-mod-3}\Ra\eqref{uc-mod-4}$ It is obvious, just take $\eta_3(q,k)=\eta_2(q,k)$ for all $q\in \Q^*_+, k\in\N$.

$\eqref{uc-mod-4}\Ra\eqref{uc-mod-3}$ For every  $r>0$, let $(q_n^r)_{n\ge 1}$ be a  nondecreasing sequence of positive rationals such that  
$\ds q_n^r\in \left(r-\frac1n,r\right]$ for all $n\ge 1$. Define 
\[\eta_2:(0,\infty)\times \N\rightarrow \N,  \quad \eta_2(r,k)=\sup\{ \eta_3(q_n^r,k)  
\mid n\ge 1\}\leq \eta_3(\lceil r\rceil, k).\]
\end{proof}

We define the theory $\cA^\omega[X,d,UCW,\eta]_{-b}$\footnote{Corrections to the definition of the theory $\cA^\omega[X,d,UCW,\eta]_{-b}$ in \cite{Leu06}: 
Proposition \ref{char-UCW-hyperbolic} is the corrected version of \cite[Proposition 3.8]{Leu06}. Furthermore, 
the axiomatization (A1)-(A3) given in this paper corrects the one from \cite{Leu06}.} of $UCW$-hyperbolic spaces as an extension of $\cA^\omega[X,d,W]_{-b}$  obtained by adding a constant  
$\eta_X$ of type $0\to 0\to 0$, together with three  axioms expressing that $\eta_X$ satisfies Proposition \ref{char-UCW-hyperbolic}.\eqref{uc-mod-4}.
\bua
\begin{array}{ll}
\!\!\!\! (A1) & \!\!\! \forall r^0,k^0\forall x^X,y^X,a^X \bigg(\lambda n^0.r>_\R 0_\R \wedge  d_X(x,a) <_\R \lambda n^0.r \wedge d_X(y,a)<_\R \lambda n^0.r\\
& \!\!\!\wedge\,\displaystyle d_X(W_X(x,y,1/2),a)>_\R \displaystyle \left(1_\R-_\R 2^{-\eta_X(r,k)}\right)\cdot_\R \lambda n^0.r\\
& \!\!\!\rightarrow \, d_X(x,y)\leq_\R 2^{-k}\cdot_\R \lambda n^0.r\bigg),\\[0.2cm]
\!\!\!\!(A2)&  \!\!\!\forall r_1^0,r_2^0,k^0\, \big(r_1\leq_\Q r_2\rightarrow \eta_X(r_1,k)\leq_0 \eta_X(r_2,k)\big),\\[0.2cm]
\!\!\!\!(A3)&  \!\!\!\forall r^0,k^0(\eta_X(r,k)=_0\eta_X(c(r),k)),
\end{array}
\eua
where  $c(n):=\min p\leq_0 n[p=_\Q n]$ is the canonical representative for rational numbers and, 
since $r^0$ codes a rational number $q$, $\lambda n^0.r$ represents $q$ as a real number
 (see \cite[Chapter 4]{Koh98} for details).

If $X$ is a nonempty set, the full-theoretic type structure 
$S^{\omega,X}:=\left<S_\rho\right>_{\rho\in {\bf T}^X}$ over $0$ and $X$ is defined by 
$\ds 
S_0:=\N$, $S_X:=X$, $S_{\rho\rightarrow \tau}:=S_\tau^{S_\rho}$,
where $S_\tau^{S_\rho}$ is the set of all set-theoretic functions $S_\rho\rightarrow S_\tau$.

Let $(X,d,W,\eta)$ be a $UCW$-hyperbolic space.  $S^{\omega,X}$ becomes a model of $\cA^\omega[X,d,UCW,\eta]_{-b}$ by 
letting the variables of type $\rho$ range over $S_\rho$, giving the natural interpretations to the constants 
of $\cA^\omega$, interpreting  $0_X$ by an arbitrary element of $X$, the constants  $d_X$ and $W_X$ as 
specified in 
\cite{Koh05} and  $\eta_X$ by $\eta_X(r, k) := \eta(c(r),k)$. We say that a sentence in the language 
${\cal L}(\mathcal{A}^\omega[X,d,UCW,\eta]_{-b})$ 
holds in a nonempty $UCW$-hyperbolic space $(X,d,W,\eta)$ if it is true in all models of 
$\mathcal{A}^\omega[X,d,UCW,\eta]_{-b}$ obtained from $S^{\omega,X}$ as above.

For  any type $\rho\in {\bf T}^X$, we define the type $\widehat{\rho}\in {\bf T}$, obtained by
replacing all occurrences of the type $X$ in $\rho$ by $0$. We say that $\rho$ has degree   $\le 1$ if $\rho=0$ or 
$\rho=0\to 0\ldots \to 0$ and  that $\rho$ has degree $1^*$ if $\widehat{\rho}$ has degree $\le 1$.  
Furthermore,  $\rho$ has degree
\be
\item $(0,X)$ if $\rho=X$ or $\rho=0\to 0 \ldots \to 0\to X$;
\item  $(1,X)$  if $\rho=X$ or  $\rho=\rho_1\ra\ldots \ra\rho_n\ra X$, 
where each $\rho_i$ has degree $\le 1$ or $(0,X)$;
\item  $(\cdot,0)$  if $\rho=0$ or  $\rho=\rho_1\ra\ldots \ra\rho_n\ra 0$;
\item  $(\cdot,X)$  if $\rho=X$ or  $\rho=\rho_1\ra\ldots \ra\rho_n\ra X$.
\ee
From now on, in order to improve readability, we shall usually omit the subscripts $_\N, _\R,_\Q,_X$ 
excepting the cases where such an omission could create confusions. We shall use  $\N$ instead of 
$0$, $\N^\N$  or $\N\to\N$ instead of $1$ and, moreover, write $n\in\N, f:\N\to\N, x\in X, T:X\to X$ instead of 
$n^0, f^1, x^X, T^{X\to X}$. 

The notion of {\em majorizability} was originally introduced by Howard \cite{How73} and subsequently 
modified by Bezem \cite{Bez85}.  Based on Bezem's notion of {\em strong majorizability} {\em s-maj}, Gerhardy and Kohlenbach 
\cite{GerKoh08} defined, for every parameter $a$ of type $X$, an {\em $a$-majorization} 
relation $\gtrsim^a_\rho$ between objects of 
type $\rho\in {\bf T}^X$ and 
their majorants of type $\widehat{\rho}\in {\bf T}$ as follows:
\be
\item ${x^*}\gtrsim^a_\N  x:\equiv  x^* \ge_\N x$ for $x,x^*\in\N$;
\item ${x^*}\gtrsim^a_X x  :\equiv (x^*)_\R\ge_\R d(x,a)$ for $x^*\in\N, x\in X$;
\item $x^* \gtrsim^a_{\rho\rightarrow\tau} x :\equiv  \forall y^*,y(y^*\gtrsim^a_{\rho} y \rightarrow 
x^* y^*  \gtrsim^a_{\tau} xy)\wedge \forall z^*,z(z^*\gtrsim^a_{\hat{\rho}} z \rightarrow x^* z^* 
\gtrsim^a_{\hat{\tau}} x^*z)$.
\ee
Restricted to the types $\bf T$, the relation $\gtrsim^a$ coincides with strong majorizability 
{\em s-maj} and, hence, for $\rho\in \bf T$ one writes {\em s-maj}$_\rho$ instead of $\gtrsim^a_\rho$, 
as in this case the parameter $a$ is irrelevant.

If $t^*\gtrsim^a t$ for terms $t^*,t$, we say that $t^*$ {\em $a$-majorizes} $t$ or that $t$ is 
{\em $a$-majorized} by $t^*$. A term $t$ is said to be {\em majorizable}  if it has an  
$a$-majorant for some $a\in X$. One can prove that  $t$ is majorizable if and only if
it has an $a$-majorant for all $a\in X$ (see, e.g., \cite[Lemma 17.78]{Koh08-book}). 

\blem\label{lemma-T-majorizable}
Let $T:X\to X$. The following are equivalent.
\be
\item $T$ is majorizable;
\item for all $x\in\N$ there exists $\,\Omega:\N\to\N$ such that 
\beq
\forall n\in\N, y\in X\bigg(d(x,y) < n \ra d(x,Ty)\leq \Omega(n)\bigg); \label{modulus-maj-x}
\eeq
\item for all $x\in\N$ there exists $\,\Omega:\N\to\N$ such that 
\beq
\forall n\in\N, y\in X\bigg(d(x,y)\le  n \ra d(x,Ty)\leq \Omega(n)\bigg) \label{modulus-maj-x-leq}.
\eeq
\ee
\elem
\begin{proof}
$T$ is majorizable if and only if $T$ is $x$-majorizable for each $x\in X$ if and only if for each 
$x\in X$ there exists a function 
$T^*:\N\to\N$ such that $T^*$ is nondecreasing and satisfies
$$\forall n\in\N\, \forall y\in X\bigg(d(x,y)\le n \ra d(x,Ty)\le T^*n\bigg).$$
$(i)\Ra(iii)$ is obvious: take $\Omega=T^*$. For the implication $(iii)\Ra(i)$, given, for 
$x\in X$, $\Omega$ satisfying 
(\ref{modulus-maj-x-leq}), define $\ds T^*n=\max_{k\le n}\Omega(k)$. \\
$(iii)\Ra(ii)$ is again obvious. For the converse implication, given $\Omega$ satisfying  
(\ref{modulus-maj-x}) define 
$\tilde{\Omega}(n)=\Omega(n+1)$. Then  $\tilde{\Omega}$ satisfies (\ref{modulus-maj-x-leq}).
\end{proof}

In the sequel, given a majorizable function $T:X\to X$ and $x\in X$, an $\Omega$ satisfying (\ref{modulus-maj-x-leq})
 will be called  a 
{\em modulus of majorizability at $x$} of $T$; we say also that $T$ is {\em x-majorizable with modulus} $\Omega$. 
We give in Lemma \ref{lemma-T-majorizable} the equivalent condition (\ref{modulus-maj-x}) for logical reasons: since $<_\R$ is a $\Sigma_1^0$ 
predicate and $\le_\R$ is a 
$\Pi_1^0$ predicate, the formula in (\ref{modulus-maj-x})  can be written in purely universal form. 

The following lemma shows that natural classes of mappings in metric or $W$-hyperbolic spaces 
are majorizable; we refer to \cite[Corollary 17.55]{Koh08-book} for the proof.

\blem\label{classes-majorizable}
Let $(X,d)$ be a metric space. 
\be
\item If $(X,d)$ is bounded with diameter $d_X$, then any function $T:X\to X$ is majorizable 
with modulus of majorizability  
$\Omega(n):=\lceil d_X\rceil$ for each $x\in X$.
\item If $T:X\to X$ is $L$-Lipschitz, then $T$ is majorizable with modulus at $x$ given by 
$\Omega(n):=n+ L^*b$, where $b,L^*\in\N$  are 
such that $d(x,Tx)\leq b$ and $L\le L^*$. In particular, any nonexpansive mapping is majorizable 
with modulus $\Omega(n):=n+b$.
\item If $(X,d,W)$ is a $W$-hyperbolic space, then any uniformly continuous mapping $T:X\to X$ 
is majorizable with modulus  
$\Omega(n):=n\cdot 2^{\alpha_T(0)}+1+b$ at $x$,  where $d(x,Tx)\leq b\in\N$ and $\alpha_T$ is 
a modulus of uniform continuity of $T$, i.e. 
$\alpha_T:\N\to\N$ satisfies
\[\forall x,y\in X\,\forall k\in\N\left(d(x,y)\leq 2^{-\alpha_T(k)}\ra d(Tx,Ty)\leq 2^{-k}\right).\]
\ee
\elem

Whenever we write $A(\ul{u})$ we mean that $A$ is a formula in our language which has only the variables $\ul{u}$ free.
A formula $A$ is called a {\em $\forall$-formula} (resp. a {\em $\exists$-formula}) if it has the form 
$$A\equiv\forall\underline{x}^{\underline{\sigma}}A_0(\underline{x},\underline{a}) \quad\text{(resp. } 
A\equiv \exists 
\underline{x}^{\underline{\sigma}}A_0(\underline{x},\underline{a})),$$
where $A_0$  is a quantifier free formula and the types in $\underline{\sigma}$ are of degree 
$1^*$ or $(1,X)$. 
We assume in the following that the constant $0_X$ does not occur in the formulas we consider; 
this is no restriction, since $0_X$ is just an arbitrary constant which could have been 
replaced by any new variable of type $X$.

The following result is an adaptation of a general logical metatheorem proved by Kohlenbach 
\cite{Koh05} for bounded $W$-hyperbolic spaces and  generalized to the unbounded 
case  by Gerhardy and Kohlenbach \cite{GerKoh08}. 

\begin{theorem}\label{logic-meta-UCW}
Let $P$ be $\N$, $\N^\N$ or $\N^{\N\times\N}$, $K$ an $\mathcal{A}^\omega$-definable compact metric space, 
$\rho$ of degree $1^*$, $B_\forall(u,y,z,n)$ a $\forall$-formula and 
$C_\exists(u,y,z,N)$ a $\exists$-formula. Assume that $\cA^\omega[X,d,UCW,\eta]_{-b}$ proves that
\bua
\forall u\in P\forall y\in K \forall z^\rho\bigg(\forall n\in\N\, B_\forall\rightarrow 
\exists N\in\N C_\exists\bigg).
\eua
Then one can extract a computable functional 
$\Phi:P\times \N^{(\N\times\ldots \times\N)}\times \N^{\N\times \N}\to\N$ 
such that the following statement is satisfied in all nonempty $UCW$-hyperbolic spaces $(X,d,W,\eta)$:

for all $z\in S_\rho, z^*\in \N^{(\N\times\ldots \times\N)}$, if there exists $a\in X$ such 
that $ z^*\gtrsim^a_\rho  z$, then 
$$\forall u\in P\forall y\in K \bigg(\forall n\le \Phi(u,z^*,\eta)\, B_\forall\ra 
\exists N\le \Phi(u,z^*,\eta)  \,C_\exists\bigg).$$
\end{theorem}
\begin{proof}
As $\leq_\R$  is purely universal and  $<_\R$ is  purely existential, one can easily see that the axioms (A1)-(A3) are universal. Furthermore, $\eta_X$ is strongly majorized by 
$\eta_X^*:=\lambda n^0,m^0.\max\{\eta_X(i,j) \mid i\leq n, j\leq m\}$. Then the proofs from \cite{GerKoh08} extend immediately to our theory (see \cite[Remark 4.13]{GerKoh08}).
\end{proof}

\bfact
\be
\item 
Instead of single premises $\forall n B_\forall$ and single variables $u,y,n$ we
may have finite conjunctions of premises as well as tuples $\ul{u}\in P, \ul{y}\in K, 
\ul{n}\in \N$ of variables. 
\item We can have also $\underline{z}^{\ul{\rho}}=z_1^{\rho_1}, 
\ldots, z_k^{\rho_k}$ for types $\rho_1,\ldots, \rho_k$ of degree $1^*$. Then   in the conclusion 
is assumed that 
$z_i^*\gtrsim^a_{\rho_i}z_i$ for one  common $a\in X$ for all $i=1,\ldots, k$. 
The bound $\Phi$ depends now on all the 
$a$-majorants $z_1^*,\ldots, z_k^*$.
\ee
\efact

The proof of Theorem \ref{logic-meta-UCW} is based on an extension to $\cA^\omega[X,d,UCW,\eta]_{-b}$ 
of Spector's \cite{Spe62} interpretation of classical analysis $\aomega$ 
using {\em bar recursion}, combined with $a$-majorization. Furthermore, the proof of the metatheorem actually 
provides an extraction algorithm for the functional $\Phi$, which  can always be defined 
in the calculus of bar-recursive functionals. However, as we shall see also in this paper,  
for concrete applications usually small fragments of $\cA^\omega[X,d,UCW,\eta]_{-b}$ are 
needed to formalize the proof. As a consequence, one gets bounds of primitive recursive 
complexity and very often  exponential  or even polynomial bounds.

We give now a very useful corollary of Theorem \ref{logic-meta-UCW}.

\bcor\label{meta-BRS-UCW}
Let $P$ be $\N$, $\N^\N$ or $\N^{\N\times\N}$, $K$ an $\aomega$-definable compact metric space, 
$B_\forall(\ul{u},\ul{y},x,x^*,T,n)$  a $\forall$-formula 
and $C_\exists(\ul{u},\ul{y},x,x^*,T,N)$ a $\exists$-formula. \\
Assume that 
$\mathcal{A}^\omega[X,d,UCW,\eta]_{-b}$ 
proves that
\[\ba{l}
\forall\,  \ul{u}\in P\, \forall\, \ul{y}\in K \, \forall \, x,x^*\in X\,\forall\, T:X\to X\,\forall\,\Omega:\N\to\N \\
\quad\quad\quad
\bigg( T \text{ is } x\text{-majorizable with modulus } \Omega\, \si\,   
\forall n\in\N \, B_\forall\,\ra \,\exists N\in\N \, C_\exists
\bigg).
\ea\]
Then one can extract a computable functional $\Phi$ satisfying the following statement for all 
$\ul{u}\in P$, $b\in\N$ and $\Omega:\N\to\N$:
\[\ba{l}
\forall\, \ul{y}\in K \, \forall \, x,x^*\in X\,\forall\, T:X\to X \\
\bigg( T \text{ is }  x\text{-majorizable with modulus } \Omega \, \si \, d(x,x^*)\leq b\, \si\,   
\forall n\le \Phi(\ul{u},b,
\Omega,\eta) \, B_\forall \\
\qquad\qquad\qquad\qquad\ra \,\exists N\le \Phi(\ul{u},b,\Omega,\eta) \, C_\exists\bigg).
\ea\]
holds in all nonempty $UCW$-hyperbolic spaces $(X,d,W,\eta)$.
\ecor
\begin{proof}
The premise "$T$ is  $x$-majorizable with modulus $\Omega$"  is a $\forall$-formula, by 
(\ref{modulus-maj-x}). Furthermore, $0$ $x$-majorizes 
$x$, $b$ is an $x$-majorant for $x^*$, since $d(x,x^*)\le b$, and $\ds T^*:=\lambda n. \max_{k\le n}\Omega(k)$ 
$x$-majorizes $T$, by the 
proof of Lemma \ref{lemma-T-majorizable}. Apply now Theorem \ref{logic-meta-UCW}.
\end{proof}

\bfact
As in the case of Theorem \ref{logic-meta-UCW}, instead of single $n\in \N$ and a single premise 
$\forall n B_\forall$ we could have 
tuples $\ul{n}=n_1,\ldots, n_k$ and a conjunction of premises $\forall n_1 B_\forall^1\si\ldots \si 
\forall n_k B_\forall^k$. 
In this case, in the premise  of the conclusion we shall have
$\forall n_1\le \Phi\,B_\forall^1\si\ldots \si \forall n_k \le \Phi\, B_\forall^k$.
\efact

If $(X,d)$ is a metric space, $C\se X$ and $T:C\rightarrow C$ is a mapping, we denote with 
$Fix(T)$ the set of fixed points of $T$. For $x\in X$ and $b,\delta>0$, let
\[Fix_{\delta}(T,x,b)=\{y\in C\mid d(y,x)\leq b \text{~and~} d(y,Ty)<\delta\}.\]
If $Fix_{\delta}(T,x,b)\ne \emptyset$ for all $\delta>0$, we say that $T$ has {\em  approximate fixed points} in a 
$b$-neighborhood of $x$.  

The following more concrete consequence of Theorem  \ref{logic-meta-UCW} shows that, under 
some conditions, the hypothesis of $T$ having fixed points can be replaced by the weaker one 
that $T$ has approximate fixed points in a $b$-neighborhood of $x$. 
Its proof is similar with the one of \cite[Corollary 4.22]{GerKoh08}.

\begin{corollary}\label{meta-Groetsch}
Let $P$ be $\N$, $\N^\N$ or $\N^{\N\times\N}$, $K$ an $\aomega$-definable compact metric space, 
$B_\forall(\ul{u},\ul{y},x,T,n)$ a $\forall$-formula and 
$C_\exists(\ul{u},\ul{y},x,T,N)$ a  $\exists$-formula. 
Assume that $\mathcal{A}^\omega[X,d,UCW,\eta]_{-b}$ proves that
\[\ba{l}
\forall \,\ul{u}\in P \,\forall\,\ul{y}\in K \, \forall  x\in X\,
\forall \, T:X\to X \, \forall \,\Omega:\N\to\N\\
\!\!\! \bigg(T  \text{ is } x\text{-majorizable with mod. } \Omega \si Fix(T)\ne\emptyset \si 
 \forall n\in\N \, B_\forall\,
\ra \,\exists N\in\N \, C_\exists\bigg).
\ea\]
It follows that one can extract a computable functional $\Phi$ such that for all \\
$\ul{u}\in P, b\in\N$ and
$\Omega:\N\to\N$,
\[\ba{l}
\forall\,\ul{y}\in K \, \forall \, x\in X\,\forall\, T:X\to X \\
\qquad\bigg( T   \text{ is } x\text{-majorizable with modulus } \Omega\,\si\, 
\forall \delta>0\big(Fix_\delta(T,x,b)\ne\emptyset\big) \, \si \\
\hfill \forall n\le \Phi(\ul{u},b,\Omega,\eta) B_\forall\,
\ra \,\exists N\le \Phi(\ul{u},b,\Omega,\eta) \, C_\exists\bigg).
\ea\]
holds in any nonempty  $UCW$-hyperbolic space $(X,d,W,\eta)$. 
\end{corollary}
\begin{proof}
The statement proved in $\mathcal{A}^\omega[X,d,UCW,\eta]_{-b}$ can be written as
\[\ba{l}
\forall\,  \ul{u}\in P\, \forall\, \ul{y}\in K \, \forall \, x,p\in X\,\forall\, T:X\to X \,\forall\, \Omega:\N\to\N\\
\bigg( T \,  x\text{-maj. mod.} \Omega \si \forall k\in\N\left(d(p,Tp)\le_\R 2^{-k}\right)
 \si  
\forall n\in\N  B_\forall \ra \exists N\in\N \, C_\exists\bigg).
\ea\]
We have used the fact that  $Fix(T)\ne\emptyset$ is equivalent with $\exists p\in X(Tp=_Xp)$ that 
is further equivalent with 
$\exists p\in X\,\forall \,k\in\N\big(d(p,Tp)\le_\R 2^{-k}\big)$, by using the definition of 
$=_X$ and $=_\R$ in our system. 
As all the premises are $\forall$-formulas, we can apply Corollary \ref{meta-BRS-UCW} to 
extract a functional $\Phi$ such that for all 
$b\in\N$,
\[\ba{l}
\forall\,  \ul{u}\in P \,\forall\, \ul{y}\in K \, \forall \, x,p\in X\,\forall\, T:X\to X\,\forall \,\Omega:\N\to\N \\
\quad\bigg(T\, x\text{-maj. mod. } \Omega \si d(x,p)\le b \si\forall k\le \Phi(\ul{u},b,\Omega,\eta)\big(d(p,Tp)\le
2^{-k}\big)\\ 
\hfill \si \, \forall n\le \Phi(\ul{u},b,\Omega,\eta) \, B_\forall\,\ra \,\exists N\le \Phi(\ul{u},b,\Omega,\eta)\, 
C_\exists\bigg),
\ea\]
that is
\[\ba{l}
\forall\,  \ul{u}\in P\,\forall\, \ul{y}\in K \, \forall \, x\in X\,\forall\, T:X\to X\,\forall \,\Omega:\N\to\N
\bigg(T \, x\text{-maj. mod. } \Omega\,\si\\
\exists p\in X\big(d(x,p)\le b \, \si\,\forall k\le \Phi(\ul{u},b,\Omega,\eta)
\big(d(p,Tp)\le 2^{-k}\big)\big)\, \si \, \\
\forall n\le \Phi(\ul{u},b,\Omega,\eta) \, B_\forall\, 
\ra \,\exists N\le \Phi(\ul{u},b,\Omega,\eta)\, 
C_\exists\bigg).
\ea\]
Use the fact that  the existence of $p\in X$ such that $d(x,p)\le b$ and $\forall k\le \Phi\big(d(p,Tp)\le 2^{-k}\big)$ 
is equivalent with 
the existence of $p\in X$ such that $d(x,p)\le b$ and $d(p,Tp)\le 2^{-\Phi}$ which is obviously implied by 
$\forall \delta>0\left(Fix_\delta(T,x,b)\ne\emptyset\right)$.
\end{proof}

We shall apply the above corollary in the next section for nonexpansive mappings $T:X\to X$. 
In this case, as we have seen in Lemma \ref{classes-majorizable}, a modulus of majorizability at $x$ is 
given by $\Omega(n)=n+\tilde{b}$, where $\tilde{b}\ge d(x,Tx)$, so the bound $\Phi$ will depend  on $\ul{u},\eta,b$ and 
$\tilde{b}>0$ such that $d(x,Tx)\le \tilde{b}$.

For all $\delta>0$ there exists $y\in X$ such that $Fix_\delta(T,x,b)\ne\emptyset$, hence
$$d(x,Tx)\le d(x,y)+d(y,Ty)+d(Ty,Tx)\le 2d(x,y)+d(y,Ty)\le 2b+\delta$$ 
for all $\delta >0$. 
It follows that $d(x,Tx)\le 2b$, so we can take $\tilde{b}:=2b$. 
As a consequence, the bound $\Phi$ will depend only on $\ul{u},b$ and $\eta$. 

\mbox{}

The above logical metatheorems were obtained for classical proofs in  metric, $W$-hyperbolic or 
$UCW$-hyperbolic spaces. Gerhardy and Kohlenbach \cite{GerKoh06} considered similar metatheorems 
for semi-intuitionistic proofs, that is proofs in intuitionistic analysis enriched with some 
non-constructive principles. 
Let ${\cal A}_i^{\omega}:=\mathbf{E-HA^\omega}+\mathbf{AC}$, where $\mathbf{E-HA^\omega}$ is 
the extensional Heyting arithmetic in all finite types and $\mathbf{AC}$ is the full axiom of choice. 
The theories $\cA_i^{\omega}[X,d]_{-b}$, $\cA_i^\omega[X,d,W]_{-b}$ 
and $\cA_i^\omega[X,d,UCW]_{-b}$ are obtained  as above as extensions of $\cA_i^{\omega}$; 
we refer to \cite{GerKoh06} for details.

Comprehension for negated formulas is the following principle:
\[
CA_\neg^{\ul{\rho}}: \quad \exists \Phi\leq_{\ul{\rho}\to \N}\lambda\, \ul{x}^{\ul{\rho}}.1\,\forall 
\ul{y}^{\ul{\rho}}\big(\Phi(\ul{y})=_\N 0\leftrightarrow \neg A(\ul{y}\big),
\]
where $\ul{\rho}=\rho_1,\ldots, \rho_k$ and $\ul{y}=y_1^{\rho_1}, \ldots, y_k^{\rho_k}$.

The following result is an adaptation to  $UCW$-spaces  of   \cite[Corollary 4.9]{GerKoh06}.

\begin{theorem}\label{cor-application-Ishikawa-intuitionistic}
Let $P$ be $\N$, $\N^\N$ or $\N^{\N\times\N}$, $K$  an $\cA_i^\omega$-definable compact 
Polish space and let $B(\ul{u}, \ul{y},T,x)$, $C(\ul{u}, \ul{y},T,x,N)$ be arbitrary formulas. 

Assume that $\cA_i^\omega[X,d,UCW,\eta]_{-b}+CA_\neg$ proves that
\begin{eqnarray*}
\forall \,\ul{u}\in P \,\forall\,\ul{y}\in K \, \forall  x\in X\,
\forall \, T:X\to X \, \forall \,\Omega:\N\to\N\\
 \bigg(T \text{ is } x\text{-majorizable with modulus }  \Omega\,\wedge\, \neg B \, \rightarrow \,\exists N\in\N\, C\bigg).
\end {eqnarray*}
Then one can extract a G\"{o}del primitive recursive   functional $\Phi$ such that for all   
$\ul{u}\in P$ and $\Omega:\N\to\N$, 
\[
\ba{c}
\forall\,\ul{y}\in K \, \forall \, x\in X\,\forall\, T:X\to X \exists N\le \Phi(\ul{u},\Omega,\eta) \\
\qquad\bigg(  T \text{ is } x\text{-majorizable with modulus } \Omega\, \si\,   \neg B \, \ra \,C\bigg).
\ea
\]
holds in any  nonempty $UCW$-hyperbolic space $(X,d,W,\eta)$.

As before, instead of  a single premise $B$, we may have a finite conjunction of premises. 
\end{theorem}

\section{Logical discussion of the  asymptotic regularity proof}\label{section-logic-discussion}

Throughout this section  $(X,d,W,\eta)$ is a $UCW$-hyperbolic space, $C\se X$ a convex subset and $T:C\to C$ is a nonexpansive mapping. The {\em Ishikawa iteration} 
starting with $x\in C$ is defined similarly with the case of normed spaces:
\[
x_0=x, \quad x_{n+1}=(1-\lambda_n)x_n\oplus\lambda_nT((1-s_n)x_n\oplus s_nTx_n),
\]
where $(\lambda_n),(s_n)$ are sequences in $[0,1]$. 

As we  explain in the sequel, the logical metatheorems  presented in the previous section guarantee that one can extract  
an effective uniform rate of asymptotic regularity for the Ishikawa iteration from the proof of the generalization of 
Theorem \ref{Ishikawa-as-reg-Banach} to $UCW$-hyperbolic spaces.

\bthm\label{Ishikawa-as-reg-UCW}
Assume that  $Fix(T)\ne\emptyset$ and that  $(\lambda_n),(s_n)$ satisfy \eqref{hyp-lambdan-sn}. Then  $\limn  d(x_n,Tx_n)=0$ for all $x\in C$.
\ethm

By an inspection of the proof of Theorem \ref{Ishikawa-as-reg-UCW}, one can see that it consists of two important 
steps. One proves first the following result.

\bprop\label{Ishikawa-liminf-xn-Txn=0}
Assume that  $Fix(T)\ne\emptyset$, $\ds \sum_{n=0}^\infty\lambda_n(1-\lambda_n)$ diverges and $\ds \lsupn s_n<1$.
Then $\ds\linfn d(x_n,Tx_n)=0$ for all $x\in C$.
\eprop

A first  remark is that the proof of Proposition \ref{Ishikawa-liminf-xn-Txn=0} is by contradiction, hence it  is ineffective. Secondly, 
it is enough to consider nonexpansive mappings $T:X\to X$, as convex subsets of $UCW$-hyperbolic spaces  are themselves  $UCW$-hyperbolic spaces.

The assumption  that  $\ds\sum_{n=0}^\infty\lambda_n(1-\lambda_n)$ diverges is equivalent with the existence of
a rate of divergence $\theta:\N\to\N$ for the series, that is a mapping $\theta$ satisfying 
$\ds\sum_{k=0}^{\theta(n)}\lambda_k(1-\lambda_k)\geq n$ for all $n\in\N$. 
As $(s_n)$ is a sequence in $[0,1]$, the assumption that $\lsupn s_n<1$ is equivalent with the existence of 
$L,N_0\in\N, L\geq 1$ such that $\ds s_n\leq 1-\frac1L$ for all $n\geq N_0$.

Furthermore, since $d(x_n,Tx_n)\geq 0$, the following statements are equivalent:
\be
\item $\ds \linfn d(x_n,Tx_n)=0$.
\item\label{modulus-liminf} for all $k,l\in \N$ there exists $N\geq k$ such that $d(x_N,Tx_N) < 2^{-l}$.
\ee
By a {\em \modliminf} $\Delta$ for $(d(x_n,Tx_n))$ we shall understand a mapping 
$\Delta:\N\times\N\to\N$  satisfying
\[\forall k,l\in\N\,\exists N\leq \Delta(l,k)\,(N\geq k \,\wedge\, d(x_N,Tx_N) < 2^{-l}).\] 
 
One can easily conclude that $\cA^\omega[X,d,UCW,\eta]_{-b}$ proves the following 
formalized version of Proposition \ref{Ishikawa-liminf-xn-Txn=0}:
\bua
\forall k,l,N_0,L\in\N \,\forall \theta:\N\to\N \,\forall\lambda^{\N\to (\N\to\N)}_{(\cdot )}, 
s^{\N\to (\N\to\N)}_{(\cdot )}\,
\forall x\in X\,\forall T:X\to X\,\\
\bigg(Fix(T)\ne\emptyset\, \wedge\, B_\forall\to  \exists N\in\N\, \big(N\geq k\,\wedge\, d_X(x_N,Tx_N) <_\R 2^{-l}
\big)\bigg),
\eua
where $\lambda^{\N\to (\N\to\N)}_{(\cdot )}, s^{\N\to (\N\to\N)}_{(\cdot )}$ represent  
elements of the compact Polish space $[0,1]^\infty$ with the product metric and
\bua
B_\forall &\equiv & T\, \text{nonexpansive}\,\wedge \, 
\forall n\in\N\left(\sum_{i=0}^{\theta(n)} \lambda_i(1-\lambda_i)\geq_\R n\right) 
\,\wedge \, \\
&& L\geq_0 1\,\wedge\,  \forall n\in\N\,\left(n\geq N_0 \to s_n\leq_\R 1-_\R \frac1L \right).
\eua
Using the representation of real numbers in our system, one can see immediately that $B_\forall$ is a universal 
formula. Corollary \ref{meta-Groetsch} and the discussion afterwards yield the extractability of a  functional 
$\Delta:=\Delta(l,\eta,b,k,N_0,L,\theta)$ such that for all $b,N_0, L, \theta, (\lambda_n), (s_n)$, 
\begin{eqnarray*}
\forall x\in X\,\,\forall T:X\to X\,\bigg(\forall \delta>0(Fix_\delta(T,x,b)\ne\emptyset)\,\si\, 
B_\forall \rightarrow \\
\qquad\forall k,l\in\N\,\exists N\leq 
\Delta\,\left(N\geq k\,\wedge\, d_X(x_N, Tx_N) <_\R 2^{-l}\right)\bigg)
\end{eqnarray*}
holds in any nonempty $UCW$-hyperbolic space $(X,d,W,\eta)$. As a consequence, $\Delta$ is a \modliminf for $(d(x_n, Tx_n))$.

The second step of the proof is the following result.

\bprop\label{Ishikawa-lim-xn-Txn=0}
Assume furthermore that $\ds\sum_{n=0}^\infty s_n(1-\lambda_n)$ converges. Then $\limn d(x_n,Tx_n)=0$.
\eprop
Let us denote $\alpha_n:=\ds\sum_{i=0}^n s_i(1-\lambda_i)$. The proof of Proposition 
\ref{Ishikawa-lim-xn-Txn=0} is fully constructive and one can easily see that 
$\cA_i^\omega[X,d,UCW,\eta]_{-b}$  proves that
\[
\forall l\in\N\,\forall \gamma:\N\to\N\,\forall\lambda_{(\cdot )}, 
s_{(\cdot )}\,
\forall x,T\,
\big(A\,\wedge\, \exists \Delta:\N\to(\N\to\N)\, B\to \exists N\in\N\, C\big),
\]
where 
\bua
C &\equiv & \forall n\in\N\left(d_X(x_{n+N},Tx_{n+N})\leq_\R 2^{-l}\right), \\
A &\equiv &  T\, \text{nonexpansive}\,\wedge \, \gamma \text{ Cauchy modulus for } (\alpha_n) \\
&\equiv & T\, \text{nonexpansive}\,\wedge \, \forall p,n\in\N\,
\left(\alpha_{\gamma(p)+n}-_\R\alpha_{\gamma(p)}\leq_\R 2^{-p}\right) \text{ and}\\
B &\equiv & \forall k,l\in\N\, \exists N\leq \Delta(l,k)\,(N\geq k \,\wedge\, d_X(x_N,Tx_N) <_\R 2^{-l}).
\eua
Let us consider the following universal formula 
\[D\equiv \forall k,l\in\N\,\exists N\leq \Delta(l,k)\,(N\geq k \,\wedge\, \widehat{d_X(x_N,Tx_N)}(l) <_\Q  
2^{-l+1}),
\]
where we refer again to \cite[Chapter 4]{Koh98} for details on the  construction $f^1\mapsto \widehat{f}$.
Since for every $l\in\N$, $\widehat{d_X(x_N,Tx_N)}(l)$ is a rational $2^{-l}$-approximation of $d_X(x_N,Tx_N)$, it is easy to see that 
$B$ implies $D$. For the same reason, $D$ implies 
\[
\forall k,l\in\N\, \exists N\leq \Delta(l+2,k)\,(N\geq k \,\wedge\, d_X(x_N,Tx_N) <_\R 2^{-l}).
\]
Thus, $\exists \Delta:\N\to(\N\to\N)B$ is equivalent to $\exists \Delta:\N\to(\N\to\N)D$, hence
$\cA_i^\omega[X,d,UCW,\eta]_{-b}$  proves that
\[
\forall l\in\N\,\forall \gamma:\N\to\N\, \forall \Delta:\N\to(\N\to\N)\forall\lambda_{(\cdot )}, 
s_{(\cdot )}\,
\forall x,T \,
\big(A\,\wedge\, D\to \exists N\in\N\, C\big).
\]
We can apply Theorem \ref{cor-application-Ishikawa-intuitionistic}  
for $T$ nonexpansive to conclude that we can 
extract a  functional $\Phi:=\Phi(l,\eta,b,\Delta,\gamma)$ such that for all $x\in X$, $T:X\to X$, 
\[\exists N\leq \Phi \left(d_X(x,Tx)\leq_\R b \,\wedge\, A\,\wedge\, D\to C\right)
\]
holds  in any nonempty $UCW$-hyperbolic space $(X,d,W,\eta)$. Thus, there exists $N\leq\Phi$ such that 
$d(x_{n+N},Tx_{n+N})\leq 2^{-l}$ for all $n\in\N$, hence $d(x_n,Tx_n)\leq 2^{-l}$ for all 
$n\geq \Phi$. It follows that $\Phi$ is a rate  of asymptotic regularity  
for the Ishikawa iteration, whose extraction is guaranteed by logical metatheorems. 
The rate $\Phi$ is also highly uniform, since it does not depend on $X,C,T,x$  except for $b$ and 
the modulus $\eta$ of uniform convexity.

That we get a full rate of asymptotic regularity $\Phi$ is a consequence of the fact that we 
treat  the constructive proof of Proposition \ref{Ishikawa-lim-xn-Txn=0} directly, by applying Theorem \ref{cor-application-Ishikawa-intuitionistic}. 
This also needs  as input a Cauchy modulus $\gamma$ for $(\alpha_n)$.
Alternatively, one can analyze the proof as a classical one and apply Theorem \ref{logic-meta-UCW} and its 
corollaries. Since the fact that $\ds\limn d(x_n,Tx_n)=0$ is a $\forall\exists\forall$-statement, one gets in this case 
only a rate of metastability (as defined by Tao \cite{Tao07}) for the sequence $(d(x_n,Tx_n))$, i.e. 
a mapping $\Psi:\N\times\N^\N\to \N$ satisfying for all $k\in\N$ and all  $g:\N\to\N$,
\[\exists N\leq \Psi(k,g)\, \forall i,j\in[N, N + g(N)]\,\left(|d(x_i,Tx_i)-d(x_j,Tx_j)|< 2^{-k}\right). \]
In order to get such a rate of metastability $\Psi$, one only 
needs a rate of metastability  for $(\alpha_n)$.

Furthermore, one can easily see that induction is used in the proof of Theorem \ref{Ishikawa-as-reg-UCW} 
only to get inequalities on $(x_n)$ (see Proposition \ref{Ishikawa-useful-prop}.\eqref{ineq}). 
However, these inequalities are universal lemmas, hence one can add them as axioms, since 
their proofs have no contribution to the  extraction of the bounds. The rest of the proof 
uses only basic arithmetic, so it can be formalized in a small fragment of  $\cA^\omega[X,d,UCW,\eta]_{-b}$. 
As a consequence, the logical metatheorems guarantee that the bound $\Phi$ is a simple 
polynomial in the input data and the unwinding of the proof, given in the next section, 
produces such a bound.

\section{The quantitative asymptotic regularity result}\label{proof-Ishikawa-main}

In the following we give a proof of the quantitative version of Theorem \ref{Ishikawa-as-reg-UCW}.

\bthm\label{Ishikawa-main}
Let $(X,d,W,\eta)$  be a $UCW$-hyperbolic space, $C\se X$  a convex subset  and $T:C\rightarrow C$  
a nonexpansive mapping. 

Assume that $(\lambda_n),(s_n)$ are sequences in $[0,1]$ satisfying the following properties
\be
\item\label{main-Ishikawa-1-lambda-theta} $\ds\sum_{n=0}^\infty\lambda_n(1-\lambda_n)$ with 
rate of divergence $\theta :\N\to\N$;
\item\label{main-Ishikawa-1-hyp-sn} $\limsup_n s_n<1$ with $L,N_0\in\N$ satisfying 
$\ds s_n\leq 1-\frac1L$ for all $n\geq N_0$;
\item\label{main-Ishikawa-sn-lambdan} $\ds\sum_{n=0}^\infty s_n(1-\lambda_n)$ converges with 
Cauchy modulus $\gamma$.
\ee
Let $x\in C,b>0$ be such that for any $\delta >0$ there is $y\in C$ with 
\begin{equation}
d(x,y)\le b \quad \mbox{and~~}  d(y,Ty) <\delta. \label{hyp-main-b-afp}
\end{equation}

Let $(x_n)$ be the Ishikawa iteration starting with $x$. Then  $\limn d(x_n,Tx_n)=0$ and moreover
\beq
\forall \eps>0\forall n\ge \Phi(\varepsilon,\eta,b,N_0,L,\theta,\gamma)\bigg(d(x_n,Tx_n)<\eps\bigg),
\eeq
where 
\beq 
\Phi:=\Phi(\varepsilon,\eta,b,N_0,L,\theta,\gamma)= \theta(P+\gamma_0+1+N_0),
\eeq
with $\ds \gamma_0= \gamma\left(\frac{\eps}{8b}\right)$ and $\ds P= \left\lceil\frac{L(b+1)}{\varepsilon\cdot \eta\left(b+1,\displaystyle\frac{\varepsilon}{L(b+1)}\right)}
\right\rceil$.
\ethm

We recall a very useful property of UCW-hyperbolic spaces.

\blem\label{prop-UCW}\cite{Leu10}
Let $(X,d,W,\eta)$ be a $UCW$-hyperbolic space. Assume that $r>0,\varepsilon\in(0,2]$ and $a,x,y\in X$ are such that $d(x,a)\le r$, $d(y,a)\le r$  and $d(x,y)\ge\eps r$.
 Then for any $\lambda\in[0,1]$ and for all $s\geq r$,
\bua
d(\lambdaxy,a) \le \left(1-2\lambda(1-\lambda)\eta\left(s,\eps\right)\right)r.
\eua
\elem

For simplicity, we use the notation $y_n := (1-s_n)x_n\oplus s_nTx_n$. The 
following lemma collects properties of the Ishikawa iteration, which will be needed in the sequel.

\blem\label{Ishikawa-useful-prop}
\be
\item\label{ineq} For all $n\in\N$ and $x,z\in X$, the following hold
\bea
(1-s_n)d(x_n,Tx_n) &\leq & d(x_n,Ty_n)\label{d-xn-Txn-d-xn-Tyn}\\
d(x_{n+1},Tx_{n+1})&\leq & (1+2s_n(1-\lambda_n))d(x_n,Tx_n)\label{Ishikawa-as-reg-base-ineq}\\
d(y_n,z) &\le & d(x_n,z)+d(z,Tz)\label{dynz-dxnz}\\
d(Ty_n,z) &\le & d(x_n,z)+2d(z,Tz) \label{dTynz-dxnz}\\
d(x_{n+1},z)& \le &  d(x_n,z)+2\lambda_n d(z,Tz)\label{dTxn1z}\\
d(x_n,z) &\le & d(x,z)+2\sum_{i=0}^{n-1}\lambda_i d(z,Tz) \\
&\le & d(x,z)+2n d(z,Tz).\label{dxn-z-dxz}
\eea
\item\label{Ishikawa-useful-prop-afpp} Assume that $x\in C, b>0$ are such that  $T$ has approximate fixed points in a $b$-neighborhood of $x$. 
Then  $d(x_n,Tx_n)\leq 2b$ for all $n\in\N$.
\ee
\elem
\begin{proof}
\be
\item
\eqref{d-xn-Txn-d-xn-Tyn} and \eqref{Ishikawa-as-reg-base-ineq} are proved in 
\cite[Lemma 4.1]{Leu10}. For the proofs of \eqref{dynz-dxnz}-\eqref{dTxn1z}, just remark that 
$d(y_n,z) \le  (1-s_n)d(x_n,z)+s_nd(Tx_n,z) \le  d(x_n,z)+s_nd(Tz,z)$, $d(Ty_n,z) \le d(y_n,z)+d(z,Tz)$ and $d(x_{n+1},z)\le  (1-\lambda_n)d(x_n,z)+\lambda_n d(Ty_n,z)$. 
An easy induction gives us \eqref{dxn-z-dxz}.
\item Let $n\in\N$. We shall prove that $d(x_n,Tx_n)\leq 2b+\eps$ for all $\eps>0$. Applying the hypothesis with $\ds \delta_n:=\frac{\eps}{4n+1}$, 
we get $z\in X$ such that $d(x,z)\leq b$ and $d(z,Tz)\leq \delta_n$. 
It follows that 
\bua
d(x_n,Tx_n) &\leq & d(x_n,z)+d(Tx_n,z)\leq d(x_n,z)+d(Tx_n,Tz)+d(z,Tz)\\
&\leq & 2d(x_n,z)+d(z,Tz)\\
&\leq & 2d(x,z)+(4n+1)d(z,Tz) \quad \text{by \eqref{dxn-z-dxz}}\\
&\leq & 2b + (4n+1)\delta_n=2b+\eps.
\eua
\ee
\end{proof}

We prove first the quantitative version of Proposition \ref{Ishikawa-liminf-xn-Txn=0}.

\bprop\label{main-Ishikawa-prop-1} 
For all $\eps>0$ and all $k\in\N$,
\[
\exists N\in[k,\Delta]\,
\big(d(x_N,Tx_N)<\varepsilon\big), \label{quant-liminf-d-xn-Tyn}
\]
where $\ds\Delta:=\Delta(\varepsilon,k,\eta,b,N_0,L,\theta)=\theta(P+k+N_0)$.
\eprop
\begin{proof}
Let $\eps>0$ and $k\in\N$.  Using \eqref{d-xn-Txn-d-xn-Tyn}  and the hypothesis \eqref{main-Ishikawa-1-hyp-sn} of Theorem \ref{Ishikawa-main}, one can easily see that it suffices
to prove that 
\beq
\exists N\in[k+N_0,\Delta] \,\left(d(x_N,Ty_N) <\frac{\eps}L\right). \label{claim-d-xn-Tyn}
\eeq

Assume by contradiction that  \eqref{claim-d-xn-Tyn} does not hold, hence  
$\ds d(x_n,Ty_n)\geq \frac{\eps}L$ for all $n\in[k+N_0,\Delta]$. Let  $\ds \delta := \frac{1}{4(\Delta+1)}$ and $z\in Fix_\delta(T,x,b)$. 
We shall use in the sequel the notation $a_n := d(x_n,z)+2d(z,Tz)$.  As a consequence of \eqref{dxn-z-dxz}, 
we get  that for all $n\in[k+N_0,\Delta]$,
\begin{eqnarray*}
a_n &\leq & d(x,z)+(2n+2)d(z,Tz)\le b+2(\Delta+1)\delta < b+1.
\end {eqnarray*}
Remark that $d(x_n,z)\leq a_n$, $d(Ty_n,z)\le a_n$ (by \eqref{dTynz-dxnz}), $\ds d(x_n,Ty_n)\ge \frac{\eps}L
\geq \frac{\eps}{L(b+1)}\cdot a_n$ and $\ds 0< \frac{\eps}{L(b+1)} \le  \frac{d(x_n,Ty_n)}{b+1}\le \frac{2a_n}{b+1} \le 2$.
Hence, we can apply Lemma \ref{prop-UCW} with $\ds r:=a_n, 
s:=b+1$ and $\ds \eps:=\frac{\eps}{L(b+1)}$ to obtain that
\begin{eqnarray*}
d(x_{n+1},z)&=&d((1-\lambda_n)x_n\oplus\lambda_nTy_n,z)\\
&\le& \left(1- 2\lambda_n(1-\lambda_n)\eta\left(b+1,\frac{\eps}{L(b+1)}\right)\right)a_n\\
&=& d(x_n,z)+2d(z,Tz) - 2\lambda_n(1-\lambda_n)\eta\left(b+1,\frac{\eps}{L(b+1)}\right)a_n.
\end {eqnarray*} 

As $\ds a_n\ge \frac{d(x_n,Ty_n)}2\ge \frac{\eps}{2L}$, we get that for all $n\in[k+N_0,\Delta]$,
\beq
d(x_{n+1},z) \le  d(x_n,z)+2d(z,Tz) - \frac{\eps}L\lambda_n(1-\lambda_n)\eta\left(b+1,\frac{\eps}{L(b+1)}\right). 
\label{Ishikawa-main-1-ineq}
\eeq 
Adding (\ref{Ishikawa-main-1-ineq}) for $n=k+N_0,\ldots,\Delta$, it follows that 
\begin{eqnarray*}
d(x_{\Delta+1},z)&\le& d(x_{k+N_0},z)+2(\Delta-k-N_0+1)d(z,Tz)-\\
&& -\frac{\eps}L\eta\left(b+1,\frac{\eps}{L(b+1)}\right)\sum_{n=k+N_0}^{\Delta}\lambda_n(1-\lambda_n).
\end {eqnarray*}
Since 
\bua
 \sum_{n=k+N_0}^{\Delta}\lambda_n(1-\lambda_n)&= &\sum_{n=0}^{\theta(P+k+N_0)}\lambda_n(1-\lambda_n)-
 \sum_{n=0}^{k+N_0-1}\lambda_n(1-\lambda_n)\\
&\ge&  (P+k+N_0)-(k+N_0)=P,
\eua
we get that
\begin{eqnarray*}
d(x_{\Delta+1},z)&\le& \ds d(x_{k+N_0},z)+2(\Delta-k-N_0+1)d(z,Tz)-\\
&&-\frac{P\eps}L\eta\left(b+1,\frac{\eps}{L(b+1)}\right)\\
&\le& \ds d(x,z)+2(\Delta+1)d(z,Tz)-\frac{P\eps}L\eta\left(b+1,\frac{\eps}{L(b+1)}\right)\text{ by \eqref{dxn-z-dxz}}\\
&\le& b+\frac12 - (b+1)<0,
\end{eqnarray*}
that is a contradiction.
\end{proof}

\mbox{}

The proof of Theorem \ref{Ishikawa-main} follows now exactly like the one of \cite[Theorem 4.7]{Leu10}. However, for the sake of completeness, we sketch it in the sequel. 
Let us denote  $\alpha_n:=\ds\sum_{i=0}^n s_i(1-\lambda_i)$.  Since, by  Lemma \ref{Ishikawa-useful-prop}.\eqref{Ishikawa-useful-prop-afpp},   $d(x_n,Tx_n) \leq 2b$,  
we get, as an application of  \eqref{Ishikawa-as-reg-base-ineq}, that for all $m,n\in\N$, 
\bua
d(x_{n+m},Tx_{n+m})&\leq & d(x_n,Tx_n)+4b(\alpha_{n+m-1}-\alpha_{n-1}).
\eua
Apply now Proposition \ref{main-Ishikawa-prop-1} to get   $N\in\N$ such that $\ds d(x_N,Tx_N)<\frac{\eps}2$ and $\ds \gamma_0+1\le N\le \Phi$. For $n\ge\Phi$ it follows that
\bua
d(x_n,Tx_n)&\le & d(x_N,Tx_N)+4b(\alpha_{N+l-1}-\alpha_{N-1}), \quad \text{where }l=n-N \\
&=&d(x_N,Tx_N)+4b\left(\alpha_{\gamma_0+q+l}-\alpha_{\gamma_0+q}\right), \,\,\, \text{where }q=N-1-\gamma_0 \\
& < & \frac{\eps}2+4b\left(\alpha_{\gamma_0+q+l}-\alpha_{\gamma_0}\right) \leq \eps,
\eua
since $\gamma$ is a Cauchy modulus for $(\alpha_n)$.

\mbox{ }

\section*{Acknowledgements:}

The author would like to thank the anonymous reviewer and Ulrich Kohlenbach for valuable comments and suggestions that led to a 
corrected and  greatly improved version of the paper.

This work was supported by a grant of the Romanian National Authority for Scientific
Research, CNCS - UEFISCDI, project number PN-II-RU-TE-2011-3-0122.

\end{document}